\documentclass[11pt]{amsart}
\usepackage{amsmath, amssymb}
\usepackage{amsfonts}
\usepackage{mathrsfs}
\usepackage[arrow,matrix,curve,cmtip,ps]{xy}

\usepackage{amsthm}

\allowdisplaybreaks

\newtheorem{theorem}{Theorem}[section]

\newtheorem{corollary}[theorem]{Corollary}

\newtheorem*{theorem*}{Theorem}
\theoremstyle{remark}
\newtheorem{remark}[theorem]{Remark}
\newtheorem{definition}[theorem]{Definition}
\newtheorem{example}[theorem]{Example}

\numberwithin{equation}{section}

\newcommand{\im}{\operatorname{im }}

\newcommand{\clspan}{\operatorname{\overline{\textnormal{span}}}}
\newcommand{\reg}{\textnormal{reg}}

\newcommand{\gae}{\lower 2pt \hbox{$\, \buildrel {\scriptstyle >}\over {\scriptstyle
\sim}\,$}}

\newcommand{\lae}{\lower 2pt \hbox{$\, \buildrel {\scriptstyle <}\over {\scriptstyle
\sim}\,$}}

\newcommand{\MU}[1]{
\setbox0\hbox{$#1$}
\setbox1\hbox{$W$}
\ifdim\wd0>\wd1 #1^{\sim} \else \widetilde{#1} \fi
}

\begin{document}
\title{Ideals in Graph Algebras}

	\author{Efren Ruiz}
        \address{Department of Mathematics\\University of Hawaii,
Hilo\\200 W. Kawili St.\\
Hilo, Hawaii\\
96720-4091 USA}
        \email{ruize@hawaii.edu}
        
         \author{Mark Tomforde}
        \address{Department of Mathematics\\University of Houston\\
Houston, Texas\\
77204- 3008, USA}
        \email{tomforde@math.uh.edu}
        
\date{\today}


\thanks{The second author was supported by a grant from the Simons Foundation (\#210035 to Mark Tomforde)}

\date{\today}

\begin{abstract}

We show that the graph construction used to prove that a gauge-invariant ideal of a graph $C^*$-algebra is isomorphic to a graph $C^*$-algebra, and also used to prove that a graded ideal of a Leavitt path algebra is isomorphic to a Leavitt path algebra, is incorrect as stated in the literature.  We give a new graph construction to remedy this problem, and prove that it can be used to realize a gauge-invariant ideal (respectively, a graded ideal) as a graph $C^*$-algebra (respectively, a Leavitt path algebra). 
\end{abstract}

\maketitle

\section{Introduction}

An important, and often quoted result, in the theory of graph $C^*$-algebras is that a gauge-invariant ideal of a graph $C^*$-algebra is isomorphic to a graph $C^*$-algebra.  Likewise, an analogous result in the theory of Leavitt path algebras states that a graded ideal of a Leavitt path algebra is isomorphic to a Leavitt path algebra.  Unfortunately, it has recently been determined that the proofs of these results in the existing literature are incorrect, and moreover, the graph constructed to realize such ideals as graph $C^*$-algebras or Leavitt path algebras, does not work as intended.  The purpose of this paper is to rectify this problem by giving a new graph construction that allows one to realize a gauge-invariant ideal in a graph $C^*$-algebra as a graph $C^*$-algebra, as well as realize a graded ideal in a Leavitt path algebra as a Leavitt path algebra.  Thus we show that the often quoted results about gauge-invariant ideals (respectively, graded ideals) being graph $C^*$-algebras (respectively, Leavitt path algebras) are indeed true --- but the graph used to realize them as such is not the graph that has been previously described in the literature.

If $E$ is a graph, then the gauge-invariant ideals of the graph $C^*$-algebra $C^*(E)$ are in one-to-one correspondence with the admissible pairs $(H,S)$, consisting of a saturated hereditary subset $H$ and a set $S$ of breaking vertices of $H$.  Likewise, for any field $K$, the graded ideals of the Leavitt path algebra $L_K(E)$ are in one-to-one correspondence with the admissible pairs $(H,S)$.  In \cite[Definition~1.4]{DHS}, Deicke, Hong, and Szyma\'nski describe a graph ${}_HE_S$ formed from $E$ and a choice of an admissible pair $(H,S)$, and in \cite[Lemma~1.6]{DHS} they claim that the graph $C^*$-algebra $C^*({}_HE_S)$ is isomorphic to the gauge-invariant ideal $I_{(H,S)}$ corresponding to $(H,S)$.  In \cite[Proposition~3.7]{ARS} it is claimed that for any field $K$ the Leavitt path algebra $L_K({}_HE_S)$ is isomorphic to the graded ideal $I_{(H,S)}$ contained in $L_K(E)$, and the proof they give is modeled after the proof of  \cite[Lemma~1.6]{DHS}.  In \cite[\S1]{Ran} Rangaswamy points out that there is an error in the proofs of both \cite[Lemma~1.6]{DHS} and  \cite[Proposition~3.7]{ARS}, and in particular the argument that the proposed isomorphism is surjective is flawed in both cases.  (According to Rangaswamy, these errors were brought to his attention by John Clark and Iain Dangerfield.)

In this paper we seek to rectify these problems.  After some preliminaries in Section~\ref{prelim-sec}, we continue in Section~\ref{problem-sec} to show that \cite[Lemma~1.6]{DHS} and \cite[Proposition~3.7]{ARS} are not true as stated, and we produce counterexamples to help the reader understand where the problem occurs.  Specifically, we exhibit a graph $E$ with an admissible pair $(H,S)$ such that the gauge-invariant ideal $I_{(H,S)}$ in $C^*(E)$ is not isomorphic to $C^*({}_HE_S)$, and the graded ideal $I_{(H,S)}$ in $L_K(E)$ is not isomorphic to $L_K({}_HE_S)$.  In Section ~\ref{new-construction-sec} we describe a new way to construct a graph from a pair $(H,S)$, and we denote this graph by $\overline{E}_{(H,S)}$.  In Section~\ref{graph-alg-sec} we prove that if $I_{(H,S)}$ is a gauge-invariant ideal of $C^*(E)$, then $I_{(H,S)}$ is isomorphic to $C^*(\overline{E}_{(H,S)})$.  In Section~\ref{LPA-alg-sec} we prove that if $I_{(H,S)}$ is a graded ideal of $L_K(E)$, then $I_{(H,S)}$ is isomorphic to $L_K(\overline{E}_{(H,S)})$.  The theorems of Section~\ref{graph-alg-sec} and Section~\ref{LPA-alg-sec} show that indeed every gauge-invariant ideal of a graph $C^*$-algebra is isomorphic to a graph $C^*$-algebra, and every graded ideal of a Leavitt path algebra is isomorphic to a Leavitt path algebra.

We mention that when $S = \emptyset$ (which always occurs if $E$ is row-finite), then our graph $\overline{E}_{(H,S)}$ is the same as the graph ${}_HE_S$ of Deicke, Hong, and Szyma\'nski.  Thus  \cite[Lemma~1.6]{DHS} and \cite[Proposition~3.7]{ARS} (and their proofs) are valid when $S = \emptyset$.

\section{Preliminaries} \label{prelim-sec}

In this section we establish notation and recall some standard definitions.

\begin{definition}
A \emph{graph} $(E^0, E^1, r, s)$ consists of a set $E^0$ of vertices, a set $E^1$ of edges, and maps $r : E^1 \to E^0$ and $s : E^1 \to E^0$ identifying the range and source of each edge.  \end{definition}

\begin{definition}
Let $E := (E^0, E^1, r, s)$ be a graph. We say that a vertex $v
\in E^0$ is a \emph{sink} if $s^{-1}(v) = \emptyset$, and we say
that a vertex $v \in E^0$ is an \emph{infinite emitter} if
$|s^{-1}(v)| = \infty$.  A \emph{singular vertex} is a vertex that
is either a sink or an infinite emitter, and we denote the set of
singular vertices by $E^0_\textnormal{sing}$.  We also let
$E^0_\textnormal{reg} := E^0 \setminus E^0_\textnormal{sing}$, and
refer to the elements of $E^0_\textnormal{reg}$ as \emph{regular
vertices}; i.e., a vertex $v \in E^0$ is a regular vertex if and
only if $0 < |s^{-1}(v)| < \infty$.  A graph is \emph{row-finite}
if it has no infinite emitters.  A graph is \emph{finite} if both
sets $E^0$ and $E^1$ are finite (or equivalently, when $E^0$ is
finite and $E$ is row-finite).
\end{definition}

\begin{definition}
If $E$ is a graph, a \emph{path} is a sequence $\alpha := e_1 e_2
\ldots e_n$ of edges with $r(e_i) = s(e_{i+1})$ for $1 \leq i \leq
n-1$.  We say the path $\alpha$ has \emph{length} $| \alpha| :=n$,
and we let $E^n$ denote the set of paths of length $n$.  We
consider the vertices in $E^0$ to be paths of length zero.  We
also let $E^* := \bigcup_{n=0}^\infty E^n$ denote the paths of
finite length in $E$, and we extend the maps $r$ and $s$ to $E^*$
as follows: For $\alpha = e_1 e_2 \ldots e_n \in E^n$ with $n\geq
1$, we set $r(\alpha) = r(e_n)$ and $s(\alpha) = s(e_1)$; for
$\alpha = v \in E^0$, we set $r(v) = v = s(v)$.  A \emph{cycle} is a path $\alpha=e_1 e_2 \ldots e_n$ with length $|\alpha| \geq 1$ and $r(\alpha) = s(\alpha)$.  If $\alpha = e_1e_2 \ldots e_n$ is a cycle, an \emph{exit} for $\alpha$ is an edge $f \in E^1$ such that $s(f) = s(e_i)$ and $f \neq e_i$ for some $i$.  
\end{definition}

\begin{definition} \label{graph-C*-def}
If $E$ is a graph, the \emph{graph $C^*$-algebra} $C^*(E)$ is the universal
$C^*$-algebra generated by mutually orthogonal projections $\{ p_v
: v \in E^0 \}$ and partial isometries with mutually orthogonal
ranges $\{ s_e : e \in E^1 \}$ satisfying
\begin{enumerate}
\item $s_e^* s_e = p_{r(e)}$ \quad  for all $e \in E^1$
\item $s_es_e^* \leq p_{s(e)}$ \quad for all $e \in E^1$
\item $p_v = \sum_{\{ e \in E^1 : s(e) = v \}} s_es_e^* $ \quad for all $v \in E^0_\textnormal{reg}$.
\end{enumerate}
\end{definition}
\begin{definition} We call Conditions (1)--(3) in Definition~\ref{graph-C*-def} the \emph{Cuntz-Krieger relations}.  Any collection $\{ S_e, P_v : e \in E^1, v \in E^0 \}$ where the $P_v$ are mutually orthogonal projections, the $S_e$ are partial isometries with mutually orthogonal ranges, and the Cuntz-Krieger relations are satisfied is called a \emph{Cuntz-Krieger $E$-family}.  For a path $\alpha := e_1 \ldots e_n$, we define $S_\alpha := S_{e_1} \ldots S_{e_n}$ and when $|\alpha| = 0$, we have $\alpha = v$ is a vertex and define $S_\alpha := P_v$.
\end{definition}

\begin{definition} \label{LPA-lin-invo-def}
Let $E$ be a graph, and let $K$ be a field. We let $(E^1)^*$
denote the set of formal symbols $\{ e^* : e \in E^1 \}$.  The \emph{Leavitt path
algebra of $E$ with coefficients in $K$}, denoted $L_K(E)$,  is
the free associative $K$-algebra generated by a set $\{v : v \in
E^0 \}$ of pairwise orthogonal idempotents, together with a set
$\{e, e^* : e \in E^1\}$ of elements, modulo the ideal generated
by the following relations:
\begin{enumerate}
\item $s(e)e = er(e) =e$ for all $e \in E^1$
\item $r(e)e^* = e^* s(e) = e^*$ for all $e \in E^1$
\item $e^*f = \delta_{e,f} \, r(e)$ for all $e, f \in E^1$
\item $v = \displaystyle \sum_{\{e \in E^1 : s(e) = v \}} ee^*$ whenever $v \in E^0_\reg$.
\end{enumerate}
\end{definition}

\begin{definition}
Any collection $\{ S_e, S_e^*, P_v : e \in E^1, v \in E^0 \}$ where the $P_v$ are mutually orthogonal idempotents, and relations (1)--(4) in Definition~\ref{LPA-lin-invo-def} are satisfied is called a \emph{Leavitt $E$-family}. For a path $\alpha = e_1 \ldots e_n \in E^n$ we define $\alpha^* := e_n^*
e_{n-1}^* \ldots e_1^*$.  We also define $v^* = v$ for all $v \in
E^0$. 
\end{definition}

If $E$ is a graph, a subset $H \subseteq E^0$ is \emph{hereditary} if whenever $e \in E^1$ and $s(e) \in H$, then $r(e) \in H$.  A hereditary subset $H$ is called \emph{saturated} if $\{ v \in E^0_\reg : r(s^{-1}(v)) \subseteq H \} \subseteq H$.  For any saturated hereditary subset $H$, the \emph{breaking vertices of $H$} are the elements of the set $$B_H := \{ v \in E^0 : |s^{-1}(v)| = \infty \text{ and } 0 < | s^{-1}(v) \cap r^{-1}(E^0 \setminus H) | < \infty \}.$$  
If $\{ s_e, p_v : e \in E^1, v \in E^0 \}$ is a Cuntz-Krieger $E$-family, then for any $v \in B_H$, we define the \emph{gap projection} corresponding to $v$ to be $$p_v^H :=  p_v - \sum_{{s(e) = v} \atop {r(e) \notin H}} s_es_e^*.$$  Likewise, if $\{e, e^*, v : e \in E^1, v \in E^0 \}$ is a Leavitt $E$-family, then for any $v \in B_H$, we define the \emph{gap idempotent} corresponding to $v$ to be $$v^H := v - \sum_{{s(e) = v} \atop {r(e) \notin H}} ee^*.$$  

If $E$ is a graph, an \emph{admissible pair} $(H,S)$ consists of a saturated hereditary subset $H$ and a subset $S \subseteq B_H$ of breaking vertices for $H$.  When working with the graph $C^*$-algebra $C^*(E)$, whenever $(H,S)$ is an admissible pair, we let $I_{(H,S)}$ denote the closed two-sided ideal of $C^*(E)$ generated by $\{ p_v : v \in H \} \cup \{ p_v^H : v \in S \}$.  It is not hard to show that in this case
\begin{align*}
I_{(H,S)} = \clspan \big( \{ s_\alpha s_\beta^* : \alpha, & \beta \in E^*, \  r(\alpha) = r( \beta) \in H \} \\
&\cup \{ s_\alpha p_v^H s_\beta^* : \alpha, \beta \in E^*, r(\alpha) = r( \beta ) = v \in S \} \big).
\end{align*}
and from this equality, one can see that $I_{(H,S)}$ is a gauge-invariant ideal in $C^*(E)$.  It was proven in \cite[Theorem~3.6]{BHRS} that every gauge-invariant ideal of $C^*(E)$ has the form $I_{(H,S)}$ for an admissible pair $(H,S)$.

When working with the Leavitt path algebra $L_K(E)$, whenever $(H,S)$ is an admissible pair, we let $I_{(H,S)}$ denote the two-sided ideal of $L_K(E)$ generated by $\{ v : v \in H \} \cup \{ v^H : v \in S \}$.  It is not hard to show that in this case
\begin{align*}
I_{(H,S)} = \operatorname{span}_K \big( \{ \alpha \beta^* : \alpha, & \beta \in E^*, \  r(\alpha) = r( \beta) \in H \} \\
&\cup \{ \alpha v^H \beta^* : \alpha, \beta \in E^*, r(\alpha) = r( \beta ) = v \in S \} \big).
\end{align*}
and from this equality, one can see that $I_{(H,S)}$ is a graded ideal in $L_K(E)$.  It was proven in 
\cite[Theorem~5.7]{Tom10} that every graded ideal of $L_K(E)$ has the form $I_{(H,S)}$ for an admissible pair $(H,S)$.

\section{Problems with Previous Results in the Literature} \label{problem-sec}

In this section we provide counterexamples to \cite[Lemma~1.6]{DHS} and \cite[Proposition~3.7]{ARS}.  Let us recall the definition of the graph ${}_HE_S$.

\begin{definition}[Definition~1.4 of \cite{DHS}] \label{Old-graph-def}
Let $E$ be a directed graph, and let $(H,S)$ be an admissible pair with $H \neq \emptyset$.  Let $\widetilde{F}_E(H,S)$ denote the collection of all paths of positive length $\alpha := e_1 \ldots e_n$ in $E$ with $s(\alpha) \notin H$, $r(\alpha) \in H \cup S$, and $r(e_i) \notin H \cup S$ for $1 \leq i \leq n-1$.  We also define 
$$F_E(H,S) = \widetilde{F}_E(H,S) \setminus \{ e \in E^1 : s(e) \in S \text{ and } r(e) \in H \}.$$  We let $\overline{F}_E(H,S)$ denote another copy of $F_E(H,S)$, and if $\alpha \in F_E(H,S)$ we write $\overline{\alpha}$ for the copy of $\alpha$ in $\overline{F}_E(H,S)$.

We then define ${}_HE_S$ to be the graph with
\begin{align*}
{}_HE_S^0 &:= H \cup S \cup F_E(H,S) \\
{}_HE_S^1 &:= \{ e \in E^1 : s(e) \in H \} \cup \{ e \in E^1 : s(e) \in S \text{ and } r(e) \in H\} \cup \overline{F}_E(H,S)
\end{align*}
and we extend $s$ and $r$ to ${}_HE_S^1$ by defining $s(\overline{\alpha}) = \alpha$ and $r(\overline{\alpha}) = r(\alpha)$ for $\overline{\alpha} \in  \overline{F}_E(H,S)$.
\end{definition}

\begin{example} \label{E-graph-ex}
Let $E$ be the graph
$$
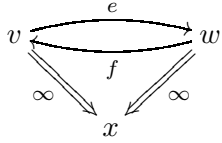
\xymatrix{
v \ar@/^/[rr]^e  \ar@{=>}[rd]_{\infty} & & w \ar@/^/[ll]^f  \ar@{=>}[ld]^{\infty} \\
& x & \\
}
$$
and let $H := \{ x \}$ and $S := \{v, w \}$.  Then $H$ is a saturated hereditary subset, $S \subseteq B_H$, and $(H,S)$ is an admissible pair.

Using Definition~\ref{Old-graph-def} we see that $F_E(H,S) = \{ e, f \}$, and the graph ${}_HE_S$ is given by
$$
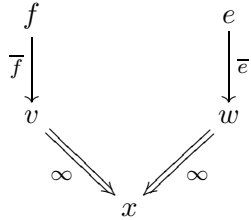
\xymatrix{
f \ar[d]_{\overline{f}} & & e \ar[d]^{\overline{e}} \\
v  \ar@{=>}[rd]_{\infty} & & w  \ar@{=>}[ld]^{\infty} \\
& x & \\
}
$$
\end{example}

\begin{remark}[Counterexample to Lemma~1.6 of \cite{DHS}] \label{C*-counterexample-rem}
ÊLet $E$ be the graph of Example~\ref{E-graph-ex}.  If we consider the graph $C^*$-algebra $C^*(E)$, and let $I_{(H,S)}$ be the gauge-invariant ideal in $C^*(E)$ corresponding to $(H,S)$, then $I_{(H,S)}$ is not isomorphic to $C^*({}_HE_S)$.  To see why this is true, suppose for the sake of contradiction, that $I_{(H,S)}$ is isomorphic to $C^*({}_HE_S)$.  Since $E$ has a finite number of vertices, $C^*(E)$ is unital.  Likewise, since ${}_HE_S$ has a finite number of vertices $C^*({}_HE_S)$ is unital, and hence $I_{(H,S)}$ is unital. Therefore $C^*(E) \cong I_{(H,S)} \oplus J$ for some ideal $J$ of $C^*(E)$.  If we let $H' := \{ v \in E^0 : p_v \in J \}$, then $H'$ is a saturated hereditary subset of $E^0$.  Since $E$ satisfies Condition~(L), the Cuntz-Krieger uniqueness theorem implies that $H'$ is nonempty.  However, any nonempty saturated hereditary subset of $E^0$ contains $x$, and thus $x \in H'$ and $p_x \in J$.  Hence $p_x \in I_{(H,S)} \cap J$, and $ I_{(H,S)} \cap J \neq 0$, contradicting the fact that $C^*(E) \cong I_{(H,S)} \oplus J$.
\end{remark}

\begin{remark}[Counterexample to Proposition~3.7 of \cite{ARS}]
Let $E$ be the graph of Example~\ref{E-graph-ex}. If $K$ is a field and we consider the Leavitt path algebra $L_K(E)$, and let $I_{(H,S)}$ be the graded ideal in $L_K(E)$ corresponding to $(H,S)$, then an argument very similar to the one given in Remark~\ref{C*-counterexample-rem} shows that $I_{(H,S)}$ is not isomorphic to $L_K({}_HE_S)$.  
\end{remark}

\begin{remark}
The only problem with the proofs of \cite[Lemma~1.6]{DHS} and \cite[Proposition~3.7]{ARS} is the verification of surjectivity.  In particular, there is an injection of $C^*({}_HE_S)$ onto a $C^*$-subalgebra of $I_{(H,S)}$ and there is an injection of $L_K({}_HE_S)$ onto a subalgebra of $I_{(H,S)}$.  However, our example shows that in general these maps are not onto.
\end{remark}

\begin{remark}
When $S = \emptyset$ (which always occurs if $E$ is row-finite), then our graph $\overline{E}_{(H,S)}$ is the same as the graph ${}_HE_S$ of Deicke, Hong, and Szyma\'nski.  Thus  \cite[Lemma~1.6]{DHS} and \cite[Proposition~3.7]{ARS} (and their proofs) are valid when $S = \emptyset$, and in particular are valid for row-finite graphs.
\end{remark}

\section{A New Graph Construction} \label{new-construction-sec}

\begin{definition}
Let $E = (E^0, E^1, r, s)$ be a graph.  If $H \subseteq E^0$ is a saturated hereditary subset of vertices and $S \subseteq B_H$ is a subset of breaking vertices for $H$, we define 
$$F_1(H,S) := \{ \alpha \in E^* : \alpha = e_1 \ldots e_n \text{ with } r(e_n) \in H \text{ and } s(e_n) \notin H \cup S \}$$
and
$$F_2(H,S) := \{ \alpha \in E^* : |\alpha| \geq 1 \text{ and } r(\alpha) \in S \}$$
Note that the paths in each of $F_1(H,S)$ and $F_2(H,S)$ must be of length $1$ or greater, and  $F_1(H,S) \cap F_2(H,S) = \emptyset$.  For $i=1,2$, let $\overline{F}_i(H,S)$ denote another copy of $F_i(H,S)$ and write $\overline{\alpha}$ for the copy of $\alpha$ in $F_i(H,S)$.

Define $\overline{E}_{(H,S)}$ to be the graph with
\begin{align*}
\overline{E}_{(H,S)}^0 &:= H \cup S \cup F_1(H,S) \cup F_2(H,S) \\
\overline{E}_{(H,S)}^1 &:= \{e \in E^1 : s(e) \in H \} \cup \{ e \in E^1 : s(e) \in S \text{ and } r(e) \in H \} \\
& \qquad \qquad    \cup \overline{F}_1(H,S) \cup \overline{F}_2(H,S)
\end{align*}
and we extend $s$ and $r$ to $\overline{E}_{(H,S)}^1$ by defining $s(\overline{\alpha}) = \alpha$ and $r(\overline{\alpha}) = r(\alpha)$ for $\overline{\alpha} \in \overline{F}_1(H,S) \cup \overline{F}_2(H,S)$.
\end{definition}

\begin{remark} \label{cycles-come-from-remark}
One very useful aspect of our construction is that the only cycles in $\overline{E}_{(H,S)}$ are paths that are also cycles in $E$.  Thus our construction does not create any additional cycles, and there is an obvious one-to-one correspondence between cycles in $E$ and cycles in $\overline{E}_{(H,S)}$.
\end{remark}

Note that if $S = \emptyset$, then $F_2(H,S) = \emptyset$ and $F_1(H,S) = \{ \alpha \in E^* : \alpha = e_1 \ldots e_n \text{ with }  r(e_n) \in H \text{ and } s(e_{n}) \notin H \}$, and in this case $\overline{E}_{(H,S)}$ is equal to ${}_HE_S$.  Thus our construction coincides with the construction of Deicke, Hong, and Szyma\'nski when $S = \emptyset$, and in particular, whenever $E$ is a row-finite graph.  When $S \neq \emptyset$, our construction is not necessarily the same, as the following example shows.

\begin{example}
Let $E$ be the graph of Example~\ref{E-graph-ex}.  Let $H := \{x \}$ and $S := \{v, w \}$.  Then $F_1(H,S) = \emptyset$ and $F_2(H,S) = \{ e, f, fe, ef, efe, fef, \ldots \}$. The graph $\overline{E}_{(H,S)}$ is given by
$$
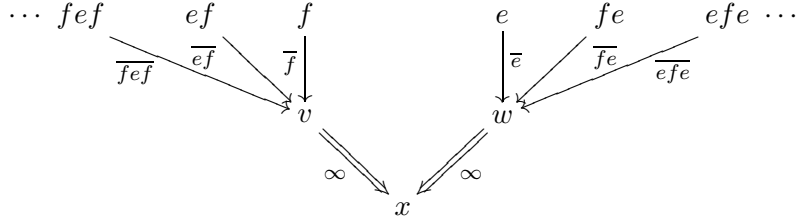
\xymatrix{
\cdots \ fef \ar[rrd]_<>(.2){\overline{fef}} &  ef \ar[rd]_<>(.0){\overline{ef}}
 & f \ar[d]_<>(.4){\overline{f}} & & e \ar[d]^<>(.4){\overline{e}} & fe \ar[ld]^<>(.0){\overline{fe}} & efe \ \cdots  \ar[lld]^<>(.2){\overline{efe}} \\
& & v  \ar@{=>}[rd]_{\infty} & & w \ar@{=>}[ld]^{\infty} & & \\
& & & x & & & \\
}
$$
and we note, in particular, that $\overline{E}_{(H,S)}$ has an infinite number of vertices.
\end{example}

\section{Ideals of Graph $C^*$-algebras} \label{graph-alg-sec}

\begin{theorem} \label{C*-ideal-is-graph-alg-thm}
Let $E$ be a graph, and suppose that $H \subseteq E^0$ is a saturated hereditary subset of vertices and $S \subseteq B_H$ is a subset of breaking vertices for $H$.  Then the gauge-invariant ideal $I_{(H,S)}$ in $C^*(E)$ is isomorphic to $C^*(\overline{E}_{(H,S)})$.
\end{theorem}

\begin{proof}
Let $\{ s_e, p_v : e \in E^1, v \in E^0 \}$ be a generating Cuntz-Krieger $E$-family for $C^*(E)$.  For $v \in \overline{E}_{(H,S)}^0$ define
$$Q_v := \begin{cases} 
p_v & \text{ if $v \in H$} \\
p_v^H & \text{ if $v \in S$} \\
s_\alpha s_\alpha^* & \text{ if $v = \alpha \in F_1(H,S)$} \\
s_\alpha p_{r(\alpha)}^H s_\alpha^*& \text{ if $v = \alpha \in F_2(H,S)$} \\
\end{cases}$$
and for $e \in \overline{E}_{(H,S)}^1$ define
$$T_e := \begin{cases} 
s_e & \text{ if $e \in E^1$} \\
s_\alpha & \text{ if $e = \overline{\alpha} \in \overline{F}_1(H,S)$} \\
s_\alpha p_{r(\alpha)}^H & \text{ if $e = \overline{\alpha} \in \overline{F}_2(H,S)$}. \\
\end{cases}$$ 
\noindent Clearly, all of these elements are in $I_{(H,S)}$.  We shall show that $\{ T_e, Q_v : e \in \overline{E}_{(H,S)}^1, v \in \overline{E}_{(H,S)}^0\}$ is a Cuntz-Krieger $\overline{E}_{(H,S)}$-family in $I_{(H,S)}$.  To begin, we see that the $Q_v$'s are mutually orthogonal projections.  In addition, the fact that the $T_e$'s are partial isometries with mutually orthogonal ranges follows from four easily verified observations:
\begin{itemize}
\item[(i)] $F_1(H,S) \cap F_2(H,S) = \emptyset$,
\item[(ii)] an element in $F_{1} (H,S)$ cannot extend an element in $F_{1} (H, S ) \cup F_{2} (H, S )$ (i.e., no element of $F_{1} (H,S)$ has the form $\alpha \beta$ for $\alpha \in F_{1} (H, S ) \cup F_{2} (H, S )$ and a path $\beta$ with $|\beta | \geq 1$),
\item[(iii)] an element in $F_{2} (H,S)$ cannot extend an element in $F_{1} (H , S )$ (i.e., no element of $F_{2} (H,S)$ has the form $\alpha \beta$ for $\alpha \in F_{1} (H, S )$ and a path $\beta$ with $|\beta | \geq 1$), and
\item[(iv)] $p_v^H s_e = 0$ whenever $e \in E^1$ with $s(e) = v$ and $r(e) \notin H$.
\end{itemize}

To see the Cuntz-Krieger relations hold, we consider cases for $e \in \overline{E}_{(H,S)}^1$.  If $e \in E^1$, then $r(e) \in H$ and
$$T_e^* T_e = s_e^*s_e = p_{r(e)} = Q_{r(e)}.$$
If $e = \overline{\alpha} \in \overline{F}_1(H,S)$, then $r(\alpha) \in H$ and
$$T_e^* T_e = T_{\overline{\alpha}}^* T_{\overline{\alpha}} = s_\alpha^*s_\alpha = p_{r(\alpha)} = Q_{r(\alpha)} = Q_{r(\overline{\alpha})} = Q_{r(e)}.$$
If $e = \overline{\alpha} \in \overline{F}_2(H,S)$, then $r(\alpha) \in S$ and
$$T_e^* T_e = T_{\overline{\alpha}}^* T_{\overline{\alpha}} = p_{r(\alpha)}^H s_\alpha^*s_\alpha p_{r(\alpha)}^H = p^H_{r(\alpha)} = Q_{r(\alpha)} = Q_{r(\overline{\alpha})} = Q_{r(e)}.$$
Thus $T_e^* T_e = Q_{r(e)}$ for all $e \in  \overline{E}_{(H,S)}^1$, and the first Cuntz-Krieger relation holds.

For the second Cuntz-Krieger relation, we again let $e \in \overline{E}_{(H,S)}^1$ and consider cases.  If $e \in E^1$ with $s(e) \in H$, then $$Q_{s(e)} T_e = p_{s(e)} s_e = s_e = T_e$$
and if $e \in E^1$ with $s(e) \in S$ and $r(e) \in H$, then $$Q_{s(e)} T_e = p^H_{s(e)} s_e = s_e = T_e.$$ If $e = \overline{\alpha} \in \overline{F}_1(H,S)$, then 
$$Q_{s(e)} T_e = Q_\alpha T_{\overline{\alpha}} = s_\alpha s_\alpha^* s_\alpha = s_\alpha = T_{\overline{\alpha}} = T_e.$$
If $e = \overline{\alpha} \in \overline{F}_2(H,S)$, then 
$$Q_{s(e)} T_e = Q_\alpha T_{\overline{\alpha}} = s_\alpha p_{r(\alpha)}^H s_\alpha^* s_\alpha p_{r(\alpha)}^H  = s_\alpha p_{r(\alpha)}^H = T_{\overline{\alpha}} = T_e.$$
Thus $Q_{s(e)} T_e = T_e$ for all $e \in \overline{E}_{(H,S)}^1$, so that $T_e T_e^* \leq Q_{s(e)}$ for all $e \in 
\overline{E}_{(H,S)}^1$, and the second Cuntz-Krieger relation holds.

For the third Cuntz-Krieger relation, suppose that  $v \in \overline{E}_{(H,S)}^0$ and that $v$ emits a finite number of edges in $\overline{E}_{(H,S)}$.  Since every vertex in $S$ is an infinite emitter in $v \in \overline{E}_{(H,S)}$, we must have $v \in H$, $v \in \overline{F}_1(H,S)$, or $v \in \overline{F}_2(H,S)$.  If $v \in H$, then the set of edges that $v$ emits in $\overline{E}_{(H,S)}$ is equal to the set of edges that $v$ emits in $E$, and hence $$Q_v = p_v = \sum_{ \{e \in E^1 : s(e) = v \} } s_es_e^* = \sum_{ \{e \in \overline{E}_{(H,S)}^1 : s(e) = v\} } T_eT_e^*.$$
If $v \in \overline{F}_1(H,S)$, then $v = \alpha$ with $r(\alpha) \in H$, and the element $\overline{\alpha}$ is the unique edge in $\overline{E}_{(H,S)}^0$ with source $v$, so that $$Q_v = s_\alpha s_\alpha^* = T_e T_e^*.$$
If $v \in \overline{F}_2(H,S)$, then $v = \alpha$ with $r(\alpha) \in S$, and the element $\overline{\alpha}$ is the unique edge in $\overline{E}_{(H,S)}^0$ with source $v$, so that $$Q_v = s_\alpha p_{r(\alpha)}^H s_\alpha^* = s_\alpha p_{r(\alpha)}^ H (s_\alpha p_{r(\alpha)}^H)^* = T_e T_e^*.$$
Thus the third Cuntz-Krieger relation holds, and $\{ T_e, Q_v : e \in \overline{E}_{(H,S)}^1, v \in \overline{E}_{(H,S)}^0\}$ is a Cuntz-Krieger $\overline{E}_{(H,S)}$-family in $I_{(H,S)}$.

If $\{ q_v, t_e : v \in \overline{E}_{(H,S)}^0, \overline{E}_{(H,S)}^1 \}$ is a generating Cuntz-Krieger $\overline{E}_{(H,S)}$-family in $C^*(\overline{E}_{(H,S)})$, then by the universal property of $C^*(\overline{E}_{(H,S)})$ there exists a $*$-homomorphism $\phi : C^*(\overline{E}_{(H,S)}) \to I_{(H,S)}$ with $\phi (q_v) = Q_v$ for all  $v \in \overline{E}_{(H,S)}^0$ and $\phi(t_e) = T_e$ for all  $e \in \overline{E}_{(H,S)}^1$. 

We shall show injectivity of $\phi$ by applying the generalized Cuntz-Krieger uniqueness theorem of \cite{Szy}.  To verify the hypotheses, we first see that if $v \in \overline{E}_{(H,S)}^0$, then $\phi( q_v) = Q_v \neq 0$.  Second, if $e_1 \ldots e_n$ is a vertex-simple cycle in $\overline{E}_{(H,S)}$ with no exits (\emph{vertex-simple} means $s(e_i) \neq s(e_j)$ for $i \neq j$), then since the cycles in $\overline{E}_{(H,S)}$ come from cycles in $E$ (see Remark~\ref{cycles-come-from-remark}), we must have that $e_i \in E^1$ for all $1 \leq i \leq n$, and $e_1 \ldots e_n$ is a cycle in $E$ with no exits.  Thus $\phi(t_{e_1 \ldots e_n}) = \phi(t_{e_1}) \ldots \phi(t_{e_n}) = s_{e_1} \ldots s_{e_n} = s_{e_1 \ldots e_n}$ is a unitary whose spectrum is the entire circle.  It follows from the generalized Cuntz-Krieger uniqueness theorem \cite[Theorem~1.2]{Szy} that $\phi$ is injective.

Next we shall show surjectivity of $\phi$.  Recall that 
\begin{align*}
I_{(H,S)} = \clspan \big( \{ s_\alpha s_\beta^* : \alpha, & \beta \in E^*, \  r(\alpha) = r( \beta) \in H \} \\
&\cup \{ s_\alpha p_v^H s_\beta^* : \alpha, \beta \in E^*, r(\alpha) = r( \beta ) = v \in S \} \big).
\end{align*}
Since $\im \phi$ is a $C^*$-subalgebra, to show $\phi$ maps onto $I_{(H,S)}$ it suffices to prove two things: (1) $s_\alpha \in \im \phi$ when $\alpha \in E^*$ with $r(\alpha) \in H$; and (2) $s_\alpha p_v^H \in \im \phi$ when $\alpha \in E^*$ with $r(\alpha) = v \in S$.  To establish (1), let $\alpha \in E^*$ with $r(\alpha) \in H$.  
Let us write $\alpha$ as a finite sequence of edges in $E^1$.  There are three cases to consider.

\noindent \textsc{Case I:} $\alpha = e_1 \ldots e_n$ with $s(e_1) \in H$.

\noindent Then $s(e_i) \in H$ and $e_i \in \overline{E}_{(H,S)}^1$ for all $1 \leq i \leq n$, so that $$\phi (T_\alpha) = \phi( T_{e_1}) \ldots \phi (T_{e_n}) = s_{e_1} \ldots s_{e_n} = s_\alpha,$$ and $s_\alpha \in \im \phi$.

\smallskip

\smallskip

\noindent \textsc{Case II:} $\alpha = f e_1 \ldots e_n$ with $r(f) = s(e_1) \in H$ and $s(f) \in S$.

\noindent Then $f \in \overline{E}_{(H,S)}^1$.  Likewise, $s(e_i) \in H$ and $e_i \in \overline{E}_{(H,S)}^1$ for all $1 \leq i \leq n$.  Hence $$\phi (T_\alpha) = \phi(T_{f}) \phi( T_{e_1}) \ldots \phi (T_{e_n}) = s_f s_{e_1} \ldots s_{e_n} = s_\alpha,$$ and $s_\alpha \in \im \phi$.

\smallskip

\smallskip

\noindent \textsc{Case III:} $\alpha = f_1 \ldots f_m e_1 \ldots e_n$ with $r(f_m) = s(e_1) \in H$ and $s(f_m) \notin H \cup S$.

\noindent Then $\beta := f_1 \ldots f_m \in F_1(H,S)$ and $e_i \in \overline{E}_{(H,S)}^1$ for all $1 \leq i \leq n$, so that $$\phi (t_{\overline{\beta}} t_{e_1} \ldots t_{e_n}) = T_{\overline{\beta}}  T_{e_1} \ldots T_{e_n} = s_\beta s_{e_1} \ldots s_{e_n} = s_\alpha$$ and $s_\alpha \in \im \phi$.

\smallskip

\smallskip

\noindent \textsc{Case IV:} $\alpha = f_1 \ldots f_m g e_1 \ldots e_n$ with $r(g) = s(e_1) \in H$, $s(g) \in S$, and $m \geq 1$.

\noindent Then $\beta := f_1 \ldots f_m \in F_2(H,S)$, $g \in  \overline{E}_{(H,S)}^1$, and $e_i \in \overline{E}_{(H,S)}^1$ for all $1 \leq i \leq n$, so that 
\begin{align*}
\phi (t_{\overline{\beta}} t_g t_{e_1} \ldots t_{e_n}) &= T_{\overline{\beta}}  T_g T_{e_1} \ldots T_{e_n} = s_\beta p_{s(g)}^H s_g s_{e_1} \ldots s_{e_n} \\
&= s_\beta s_g s_{e_1} \ldots s_{e_n} = s_\alpha.
\end{align*}
Thus $s_\alpha \in \im \phi$.

We have shown that (1) holds.  To show (2), let $\alpha \in E^*$ with $r(\alpha) \in S$.  If $|\alpha| = 0$, then $\alpha = v \in S$, and $\phi( q_v) = Q_v = p_v^H = s_\alpha p_v^H$, so $s_\alpha p_v^H \in \im \phi$.  If $|\alpha| \geq 1$, then $\alpha \in F_2(H,S)$ and $\phi(t_\alpha) = T_\alpha = s_\alpha p_{r (\alpha)}^H$, so that $s_\alpha p_{r (\alpha)}^H \in \im \phi$.  Hence (2) holds, and $\phi$ is surjective.  Thus $\phi$ is the desired $*$-isomorphism from $C^*(\overline{E}_{(H,S)})$ onto $I_{(H,S)}$.
\end{proof}

\begin{corollary} \label{ideals-in-graph-C*-are-graph-C*-cor}
If $E$ is a graph, then every gauge-invariant ideal of $C^*(E)$ is isomorphic to a graph $C^*$-algebra.
\end{corollary}

\begin{proof}
It follows from \cite[Theorem~3.6]{BHRS} that every gauge-invariant ideal of $C^*(E)$ has the form $I_{(H,S)}$ for some saturated hereditary subset $H \subseteq E^0$ and some subset of breaking vertices $S \subseteq B_H$.
\end{proof}

\begin{corollary}
If $E$ is a graph that satisfies Condition~(K), then every ideal of $C^*(E)$ is isomorphic to a graph $C^*$-algebra.
\end{corollary}

\begin{proof}
It follows from \cite[Theorem~3.6]{BHRS} and \cite[Theorem~3.5]{DT1} that if the graph $E$ satisfies Condition~(K), then every ideal in $C^*(E)$ is gauge invariant.
\end{proof}

\section{Ideals of Leavitt Path Algebras} \label{LPA-alg-sec} 

Using techniques similar to that of Section~\ref{graph-alg-sec} we are able to obtain analogous results for Leavitt path algebras.

\begin{theorem} \label{LPA-ideal-is-graph-alg-thm}
Let $E$ be a graph and let $K$ be a field.  Suppose that $H \subseteq E^0$ is a saturated hereditary subset of vertices and $S \subseteq B_H$ is a subset of breaking vertices for $H$.  Then the graded ideal $I_{(H,S)}$ in $L_K(E)$ is isomorphic to $L_K(\overline{E}_{(H,S)})$.
\end{theorem}

\begin{proof}
The proof of this result is very similar to the proof of Theorem~\ref{C*-ideal-is-graph-alg-thm}, using \cite[Theorem~3.7]{AMMS} in place of  \cite[Theorem~1.2]{Szy}, so we will only sketch the argument: Let $\{ e, e^*, v : e \in E^1, v \in E^0 \}$ be a generating Leavitt $E$-family for $L_K(E)$.  For $v \in \overline{E}_{(H,S)}^0$ define
$$Q_v := \begin{cases} 
v & \text{ if $v \in H$} \\
v^H & \text{ if $v \in S$} \\
\alpha \alpha^* & \text{ if $v = \alpha \in F_1(H,S)$} \\
\alpha {r(\alpha)}^H \alpha^*& \text{ if $v = \alpha \in F_2(H,S)$} \\
\end{cases}$$
and for $e \in \overline{E}_{(H,S)}^1$ define
$$T_e := \begin{cases} 
e & \text{ if $e \in E^1$} \\
\alpha & \text{ if $e = \overline{\alpha} \in \overline{F}_1(H,S)$} \\
\alpha {r(\alpha)}^H & \text{ if $e = \overline{\alpha} \in \overline{F}_2(H,S)$}. \\
\end{cases}$$ and $$T_e^* := \begin{cases} 
e^* & \text{ if $e \in E^1$} \\
\alpha^* & \text{ if $e = \overline{\alpha} \in \overline{F}_1(H,S)$} \\
{r(\alpha)}^H \alpha^*  & \text{ if $e = \overline{\alpha} \in \overline{F}_2(H,S)$}. \\
\end{cases} 
$$

\noindent Clearly, all of these elements are in $I_{(H,S)}$, and using an argument nearly identical to that in the proof of Theorem~\ref{C*-ideal-is-graph-alg-thm} we obtain that $\{ T_e, T_e^*, Q_v : e \in \overline{E}_{(H,S)}^1, v \in \overline{E}_{(H,S)}^0\}$ is a Leavitt $\overline{E}_{(H,S)}$-family in $I_{(H,S)}$.  

The universal property of $L_K(\overline{E}_{(H,S)})$ gives an algebra homomorphism $\phi : L_K(\overline{E}_{(H,S)}) \to I_{(H,S)}$, one may apply \cite[Theorem~3.7]{AMMS} to obtain injectivity of $\phi$, and surjectivity of $\phi$ follows from an argument similar to the surjectivity argument in the proof of Theorem~\ref{C*-ideal-is-graph-alg-thm}.  Hence $\phi$ is the desired algebra isomorphism from $L_K(\overline{E}_{(H,S)})$ onto $I_{(H,S)}$.
\end{proof}

\begin{corollary} \label{ideals-in-LPA-are-LPA-cor}
If $E$ is a graph and $K$ is a field, then every graded ideal of $L_K(E)$ is isomorphic to a Leavitt path algebra.
\end{corollary}

\begin{proof}
It follows from \cite[Theorem~5.7]{Tom10} that every graded ideal of $L_K(E)$ has the form $I_{(H,S)}$ for some saturated hereditary subset $H \subseteq E^0$ and some subset of breaking vertices $S \subseteq B_H$.
\end{proof}

\begin{corollary}
If $E$ is a graph that satisfies Condition~(K), and $K$ is a field, then every ideal of $L_K(E)$ is isomorphic to a Leavitt path algebra.
\end{corollary}

\begin{proof}
It follows from \cite[Theorem~6.16]{Tom10} that if the graph $E$ satisfies Condition~(K), then every ideal in $L_K(E)$ is graded.
\end{proof}

\section{Concluding Remarks} \label{concluding-sec} 

We have looked at the results in the prior literature that have used \cite[Lemma~1.6]{DHS} and \cite[Proposition~3.7]{ARS} and the construction of the graph ${}_HE_S$ to see how these results and their proofs are affected when one instead uses Theorem~\ref{C*-ideal-is-graph-alg-thm} and Theorem~\ref{LPA-ideal-is-graph-alg-thm} and the new construction $\overline{E}_{(H,S)}$.  Fortunately, all the results we found are still true as stated, and in many cases the proofs remain unchanged or need only slight modifications.  (We warn the reader there are undoubtedly some uses of the results we have missed, and so one should check carefully any time \cite[Lemma~1.6]{DHS} and \cite[Proposition~3.7]{ARS} are applied in the literature.)  We summarize our findings here.

We have found uses of \cite[Lemma~1.6]{DHS} and  \cite[Proposition~3.7]{ARS} in the following papers: \cite{AP}, \cite {ARS}, \cite{DHS}, \cite{ET},  and \cite{RT}.  We go through the uses in each of these papers to describe how the conclusions of the results still hold and what modifications are needed in the proofs.

In \cite{AP} the realization of a graded ideal as a Leavitt path algebra given in \cite[Lemma~1.2]{AP} is incorrect.  However, all uses of \cite[Lemma~1.2]{AP} throughout \cite{AP} may be replaced by Theorem~\ref{LPA-ideal-is-graph-alg-thm}.  In particular, when \cite[Lemma~1.2]{AP} is applied in \cite[Corollary~1.4]{AP},  the only fact that is needed is that a graded ideal in a Leavitt path algebra is itself isomorphic to a Leavitt path algebra, which is still true by Corollary~\ref{ideals-in-LPA-are-LPA-cor}.  In addition, when \cite[Lemma~1.2]{AP} is applied in the proof of \cite[Lemma~2.3]{AP}, one can verify that the argument used will still hold if one instead applies Theorem~\ref{LPA-ideal-is-graph-alg-thm} and uses our graph construction $\overline{E}_{(H,S)}$.

In \cite{APS} the realization of a graded ideal as a Leavitt path algebra given in \cite[Lemma~5.2]{APS} is incorrect.  However, all uses of \cite[Lemma~5.2]{APS} throughout \cite{APS} may be replaced by Theorem~\ref{LPA-ideal-is-graph-alg-thm}.   In particular, \cite[Lemma~5.3]{APS} only requires that any graded ideal in a Leavitt path algebra is itself isomorphic to a Leavitt path algebra, which is still true by Corollary~\ref{ideals-in-LPA-are-LPA-cor}.  In addition, when \cite[Lemma~5.2]{APS} is applied in the proof of \cite[Proposition~5.4]{APS}, one can verify that the argument used will still hold if one instead applies Theorem~\ref{LPA-ideal-is-graph-alg-thm} and uses our graph construction $\overline{E}_{(H,S)}$.

In \cite {ARS} the realization of a graded ideal as a Leavitt path algebra given in \cite[Proposition~3.7]{ARS} is incorrect.  However, this result is only applied in the proofs of \cite[Proposition~3.8]{ARS} and \cite[Theorem~3.15]{ARS}.  Both of these results only require that any graded ideal in a Leavitt path algebra is itself isomorphic to a Leavitt path algebra, a fact that is still true by Corollary~\ref{ideals-in-LPA-are-LPA-cor}.  

In \cite{DHS} the realization of a gauge-invariant ideal given in \cite[Lemma~1.6]{DHS} is incorrect.  However, \cite[Lemma~1.6]{DHS} is only used in two places in \cite{DHS}.  The first is \cite[Lemma~3.2]{DHS}, where the construction of ${}_HE_S$ is used for an ideal of the form $I_{(H, \emptyset)}$.  Since $S= \emptyset$, and the constructions for the graphs ${}_HE_S$ and $\overline{E}_{(H,S)}$ agree in this case, one may simply replace the application of \cite[Lemma~1.6]{DHS} by an application of Theorem~\ref{C*-ideal-is-graph-alg-thm} and the rest of the argument remains valid.  The second use of \cite[Lemma~1.6]{DHS} is in the proof of \cite[Theorem~3.4]{DHS} where the authors use the fact that all cycles in  ${}_HE_S$ come from $E$.  This is still true for  $\overline{E}_{(H,S)}$ (see Remark~\ref{cycles-come-from-remark}) and the same arguments hold when using $\overline{E}_{(H,S)}$.

In \cite{ET} the result of \cite[Lemma~1.6]{DHS} is applied in the proofs of \cite[Proposition~6.4]{ET}, \cite[Remark~3.2]{ET}, and \cite[Theorem~4.5]{ET}.  In \cite[Proposition~6.4]{ET} the construction of ${}_HE_S$ is used for an ideal of the form $I_{(H, \emptyset)}$ and since $S= \emptyset$ the constructions for the graphs ${}_HE_S$ and $\overline{E}_{(H,S)}$ agree in this case, the proof of \cite[Proposition~6.4]{ET} remains valid.  In the proofs of \cite[Remark~3.2]{ET} and \cite[Theorem~4.5]{ET}, the authors only need the fact that any gauge-invariant ideal in a graph $C^*$-algebra is isomorphic to a graph $C^*$-algebra, which is still true by Corollary~\ref{ideals-in-graph-C*-are-graph-C*-cor}.  

In \cite{RT} the the result of \cite[Lemma~1.6]{DHS} is applied in the proof of \cite[Lemma~3.11]{RT}.  However, all that is needed is that any gauge-invariant ideal in a graph $C^*$-algebra is isomorphic to a graph $C^*$-algebra, which is still true by Corollary~\ref{ideals-in-graph-C*-are-graph-C*-cor}.

\begin{remark}
The result of Theorem~\ref{C*-ideal-is-graph-alg-thm} is a nice complement to the fact that every quotient of a graph $C^*$-algebra by a gauge-invariant ideal is isomorphic to a graph $C^*$-algebra (see \cite[Corollary~3.5]{BHRS}).   Combining these results shows that the class of graph $C^*$-algebras is closed under passing to gauge-invariant ideals and taking quotients by gauge-invariant ideals.  In particular, if $E$ is a graph satisfying Condition~(K), then every ideal of $C^*(E)$ is isomorphic to a graph $C^*$-algebra, and every quotient of $C^*(E)$ is isomorphic to a graph $C^*$-algebra.

In a very similar way, the result of Theorem~\ref{LPA-ideal-is-graph-alg-thm} is a nice complement to the fact that every quotient of a Leavitt path algebra by a graded ideal is isomorphic to a Leavitt path algebra (see \cite[Theorem~5.7]{Tom10}).   Combining these results shows that the class of Leavitt path algebras is closed under passing to graded ideals and by taking quotients by graded ideals, and if $E$ is a graph satisfying Condition~(K), then every ideal of $L_K(E)$ is isomorphic to a Leavitt path algebra, and every quotient of $L_K(E)$ is isomorphic to a Leavitt path algebra.
\end{remark}

\begin{remark}
All ideals of graph $C^*$-algebras that we consider in this paper are gauge invariant, and likewise all ideals of Leavitt path algebras that we consider are graded.  We mention that there has been significant work done analyzing general ideals in graph $C^*$-algebras and Leavitt path algebras.  In particular, Hong and Szyma\'nski have given a description of the (not necessarily gauge-invariant) primitive ideals of a graph $C^*$-algebra in \cite{HS}.  Likewise, the (not necessarily graded) prime ideals of a Leavitt path algebra have been described by Aranda Pino, Pardo, and Siles Molina in \cite{APM}.
\end{remark}

\end{document}